\documentclass[reqno, 11pt]{amsart}
\usepackage[percent]{overpic}
\usepackage{calc,graphicx,amsfonts,amsthm,amscd,amsmath,amssymb,enumerate,dsfont,mathrsfs,paralist}
\usepackage{subfig}
\usepackage{pdfsync}
\usepackage{stmaryrd}
\usepackage[numeric,initials,nobysame]{amsrefs}
\usepackage{marginnote}
\usepackage[top=2.5cm, bottom=2.5cm, outer=2.5cm, inner=2.5cm, heightrounded, marginparwidth=1.9cm, marginparsep=0.3cm]{geometry}

\usepackage[colorlinks=true, pdfstartview=FitV, linkcolor=blue, citecolor=blue, urlcolor=blue,pagebackref=false]{hyperref}

\usepackage{booktabs}

\newcommand{\ie}{\hbox{\it i.e.\ }}

\reversemarginpar
\newlength\fullwidth
\numberwithin{equation}{section}

\DeclareMathSymbol{\leqslant}{\mathalpha}{AMSa}{"36} 
\DeclareMathSymbol{\geqslant}{\mathalpha}{AMSa}{"3E} 
\DeclareMathSymbol{\eset}{\mathalpha}{AMSb}{"3F}     
\renewcommand{\leq}{\;\leqslant\;}                   

\newcommand{\eps}{\varepsilon}

\def\1{\ifmmode {1\hskip -3pt \rm{I}} \else {\hbox {$1\hskip -3pt \rm{I}$}}\fi}

\newcommand{\tmix}{T_{\rm mix}}
\newcommand{\trel}{T_{\rm rel}}

\newcommand{\D}{\Delta}

\renewcommand{\l}{\lambda}

\renewcommand{\l}{\lambda}
\renewcommand{\a}{\alpha}
\renewcommand{\d}{\delta}
\renewcommand{\t}{\tau}

\newcommand{\g}{\gamma}
\newcommand{\G}{\Gamma}

\newcommand{\e}{\varepsilon}

\renewcommand{\o}{\omega}
\renewcommand{\O}{\Omega}

\newcommand{\Front}{{\rm F}}

\newcommand{\tc}{\thinspace |\thinspace}

\newtheorem{theorem}{Theorem}[section]
\newtheorem{lemma}[theorem]{Lemma}
\newtheorem{proposition}[theorem]{Proposition}
\newtheorem{corollary}[theorem]{Corollary}
\newtheorem{remark}[theorem]{Remark}

\newtheorem{claim}[theorem]{Claim}
\newtheorem{definition}[theorem]{Definition}
\newtheorem{maintheorem}{Theorem}

\newtheorem*{question*}{Question}

\newtheorem*{remark*}{Remark}
\newtheorem*{idefinition*}{Definition}

\newcommand{\N}{\mathbb N}

\newcommand{\cB}{\ensuremath{\mathcal B}}
\newcommand{\cC}{\ensuremath{\mathcal C}}

\newcommand{\cE}{\ensuremath{\mathcal E}}
\newcommand{\cF}{\ensuremath{\mathcal F}}
\newcommand{\cG}{\ensuremath{\mathcal G}}

\newcommand{\cI}{\ensuremath{\mathcal I}}

\newcommand{\cL}{\ensuremath{\mathcal L}}
\newcommand{\cM}{\ensuremath{\mathcal M}}
\newcommand{\cN}{\ensuremath{\mathcal N}}

\newcommand{\cP}{\ensuremath{\mathcal P}}

\newcommand{\cS}{\ensuremath{\mathcal S}}

\newcommand{\cX}{\ensuremath{\mathcal X}}

\newcommand{\bbE}{{\ensuremath{\mathbb E}} }
\newcommand{\bbF}{{\ensuremath{\mathbb F}} }

\newcommand{\bbI}{{\ensuremath{\mathbb I}} }

\newcommand{\bbN}{{\ensuremath{\mathbb N}} }

\newcommand{\bbP}{{\ensuremath{\mathbb P}} }

\newcommand{\bbR}{{\ensuremath{\mathbb R}} }

\newcommand{\bbZ}{{\ensuremath{\mathbb Z}} }
\newcommand{\Z}{{\ensuremath{\mathbb Z}} }

\newcommand{\E}{\mathbb{E}}

\let\a=\alpha    \let\d=\delta  \let\e=\varepsilon
 \let\g=\gamma \let\h=\eta      \let\l=\lambda
      \let\o=\omega      
  \let\s=\sigma \let\t=\tau   
  
\let\D=\Delta   \let\G=\Gamma   
\let\O=\Omega      

\renewcommand{\le}{\leq}

\begin{document}
\title[Upper triangular matrix walk]{Upper triangular matrix walk: Cutoff for finitely many columns}
\author{Shirshendu Ganguly}
\author{Fabio Martinelli}
 \thanks{Department of Statistics, University of California, Berkeley; sganguly@berkeley.edu}
 \thanks{Dip. Matematica \& Fisica, Universit\'a Roma Tre, Italy; martin@mat.uniroma3.it}
\begin{abstract}We consider random walk on the group of uni-upper triangular matrices with entries in $\bbF_2$ which forms an important example of a nilpotent group.  Peres and Sly \cite{peressly} proved tight bounds on the mixing time of this walk up to constants. It is well known that the single column projection of this chain is the one dimensional East process. In this article, we complement the Peres-Sly result by proving a cutoff result for the mixing of finitely many columns in the upper triangular matrix walk at the same location as the East process of the same dimension. Moreover, we also show that the spectral gaps of the matrix walk and the East process are equal.
The proof of the cutoff result is based on a recursive argument which uses a local version of a dual process appearing in \cite{peressly}, various combinatorial consequences of mixing and concentration results for the movement of the front in the one dimensional East process. \end{abstract}
\maketitle

\section{Introduction}
Random walks on the group $\cM_n(q)$ of $n\times n$-upper triangular matrices with ones along the diagonal and entries from the finite field $\mathbb{F}_{q}$, $q$ a prime, have received quite a lot of attention.  The random walk, sometimes called the \emph{upper triangular matrix walk}, in the case $q=2$ is the Markov chain whose generic step consists in choosing uniformly  a row among the first $(n-1)$ rows and adding to it the next row mod 2. It is easy to check that this chain is reversible w.r.t. the uniform measure on $\cM_n(2)$. A natural variant, called the \emph{lazy upper triangular matrix walk}, entails to perform the above addition with probability $1/2$.  This is equivalent to adding the  to-be-added row only after multiplying by a uniform Bernoulli variable. We refer the interested reader to  \cite{peressly} for a nice review of the literature related to this walk and other related variants. 

In this article, we consider the continuous time version of the upper triangular matrix walk  where each row at rate one  is updated by adding the row below it with probability $1/2.$ Sharp bounds on the spectral gap were proven by Stong \cite{Stong} implying, in particular, that the spectral gap $\lambda_2(n)$ is positive uniformly in $n$. Using an elegant argument, Peres and Sly \cite{peressly} proved that the total variation mixing time $t_{\rm{mix}}(n)=\Theta(n)$. 
From the above results it follows that $\lim_{n\to \infty}\lambda_2(n)
\times t_{\rm{mix}}(n)= +\infty$, a known necessary condition for the
occurrence of the so called \emph{mixing time cutoff} \cite{AD86}, \ie
a sharp transition in the total variation distance from equilibrium
which drops from being close to one to being close to zero in a very
small time window compared to the mixing time scale.    

In \cite{peres1}, Y. Peres conjectured that, for many natural classes of reversible Markov chains, the above condition (sometimes referred to as the \emph{product condition}) is also sufficient for the occurrence of cutoff, despite of the fact that, in full generality, this is known to be false (cf. \cite{LPW}*{Chapter 18}). Thus it is a natural and interesting problem  to decide whether the upper triangular matrix walk exhibits cutoff or not. 

It has been observed before and  was crucially used in \cite{peressly}, in the proof of their mixing time result, that the marginal process on a given column coincides with the East process \cites{AD02,CMRT,BFMRT} at density $1/2$, a well known constrained interacting particle system. The East chain is known to exhibit cutoff (cf. \cite{GLM}), a result which, combined with the previous observation, suggests that the upper triangular matrix walk in continuous time might do the same. 

In this paper we extend and complement the Peres-Sly result by proving that (i) the spectral gap of the upper triangular matrix walk is equal to the spectral gap of the East process on $n-1$ vertices; (ii) the marginal chain on \emph{finitely} many columns exhibits mixing time cutoff at the mixing time of the column with the largest index, among the chosen ones. 

Note that, even though at stationarity the elements of the matrix are independent random bits, the dynamics of the walk makes the column evolutions highly correlated so that the joint mixing of a group of columns is by no means a straightforward consequence of the mixing of only one column (the East process). We also remark that, perhaps surprisingly, certain numerical experiments suggest that the mixing time of the full chain is strictly larger (at the linear scale, $n$) than the mixing time of one column (see Section \ref{numerics}).  We conclude by mentioning that recent progress on the temporal mixing of other finite dimensional statistics in a similar setting was also made in \cite{dia2}. 
\section{Setup, motivations and main results }
\label{sec:prelims}
\subsection{Setting and notation} 
Throughout this article for any $n\in \bbN,$ we will write $\cX_n$ for
the vector space $\bbF_2^n,$ with the usual addition operation mod $2$ denoted by ``$+$'' in the sequel,
and standard basis vectors ${\textbf e}_1,\ldots {\textbf e}_n$. Elements of
$\cX_n$ will usually be denoted by ${\textbf Y},{\textbf Z}$ etc with
coordinates $\{{\textbf Y}(x)\}_{x=1}^n$ and they will
always be thought as column vectors. Sometimes a vector $\textbf X$ will depend on a time parameter $t$ and randomness variable $\o$. In that case we will write ${\textbf X}(t;\o;x)$ for its $x$-coordinate.   Row vectors will  always be
denoted by ${\textbf Y}^\top,{\textbf Z}^\top$ etc and their vector space will be
denoted by $\cX_n^\top$.  Multiplication between ${\textbf Z}^\top$ and ${\textbf Y}$
will be denoted by ${\textbf Z}^\top \cdot Y$. 

We will write $[n]$
for the set $\{1,2,\dots,n\}$ and we will use the letter $c$ to denote
a universal numerical
constant whose value might change from line to line inside the same
proof. Finally for any random variable $\xi$ with distribution $\cP$ and any other law $\nu$ we will sometimes write $d_{TV}(\xi,\nu)$ for the total variation distance between $\cP$ and $\nu$.
\subsubsection{The Markov chain} 
Our state space is  the set of all
upper triangular matrices with entries in $\bbF_2$:
$$
\cM_n =\{M\in \bbF_2^{n \times n}: M(i,j)=0, \text{ for } i>j,
M(i,i)=1\},
$$
\ie
$$
M=\left[\begin{array}{cccc}
1 & M(1,1) & M(1,2)&  \ldots \\
0 & 1 & \ldots & \ldots \\
\vdots & \vdots & \ldots & M(n-1,n) \\
0& 0& \ldots & 1
\end{array}\right].
$$
When $n=3,$ this is well known as the  Heisenberg group. Let $E^{i,j}$ be
the matrix which is one at the position $(i,j)$ and zero every where
else. 
On $\cM_n$ we consider the continuous time Markov chain (random walk) $\{M(t)\}_{t\ge 0}$=
$\{M^{(n)}(t)\}_{t\ge 0}$ (we will drop the dependence on $n$ in the
notation whenever there is no scope of confusion) whose generator is
given by 
\begin{equation}
 \label{eq:gen} 
(\cL f)(M)= \frac 12 \sum_{i=1}^{n-1}\sum_{a=0}^1\left[f((\bbI + a E^{i,i+1})M)-f(M)\right].
\end{equation}
Notice that for ${\textbf Y}\in \cX_n$ and $x\in [n]$: 
\begin{align*}
  ((\bbI+ E^{i,i+1}){\textbf Y})(x) &= {\textbf Y}(x) + \d_{i,x}{\textbf Y}(i+1),\\
({\textbf Y}^\top (\bbI+E^{i,i+1}))(x) &= {\textbf Y}(x) + \d_{x,i+1}{\textbf Y}(i).
\end{align*} 
Thus the Markov chain can be informally described as follows: 
every row but the last one updates itself at rate one by first sampling a fair
Bernoulli variable $\xi$ and then by adding (mod $2$) to itself the next
row 
multiplied by
$\xi$. Clearly if $\xi=0$ nothing happens, while if $\xi=1$ all the entries of the updating row, with the entry below them equal to $1,$ change values ($1\to 0$ or $0\to 1$).
For some purposes it will be  convenient to write 
$M(t)\equiv [{\textbf M}_1(t), {\textbf M}_2 (t), \ldots,{\textbf M}_n (t)]$ where
${\textbf M}_i(t)\in \cX_n$ denotes the $i^{th}$ column at time
$t$. In the sequel we will sometimes refer
to the Markov chain as the \emph{matrix walk} (MW).
\begin{remark}\label{rem:1}\ 
\begin{enumerate}
\item It is easy to check that the Markov chain is ergodic and
  reversible w.r.t. to the uniform measure $\pi=\pi^{(n)},$ on
  $\cM_n$. In particular at equilibrium all the entries of $M$ above
  the diagonal are i.i.d Bernoulli$(1/2)$ random variables.
\item The marginal process of every column is also a Markov chain and on the $i^{th}$ column it
  coincides with the East process \footnote{In the standard definition
    of the East process on $[n]$ 
    at density $p\in (0,1)$, for every $x\in [n]$ with rate one
    the spin $\s(t;x)\in \{0,1\}$, is replaced by a fresh Bernoulli($p$) variable
    iff $x=1$ or $x\in [2,n]$ \emph{and} the constraint $\s(t;x-1)=0$ is satisfied. It is easy to
  check that for $p=1/2$ and after the change of variables
  $\eta(x)=1-\s(x)$ the generator of the East process coincides with
  that of the ${\textbf M}^{(n)} (t)$ process.} on $[i-1]$ at density $1/2$, a
  well known kinetically constrained interacting particle system
  (cf. \cite{CMRT} and references therein). 
\end{enumerate}
\end{remark}

\subsection{Mixing time and open problems}
It is natural to consider the $\eps$-mixing time of the process
$\{M(t)\}_{t\ge 0}$ 
\[
\tmix(n,\eps)=\inf\Big\{t:\ \max_{M\in \cM_n} d_{TV}(M(t),\pi)\le
\eps\Big\},\quad \eps\in (0,1),
\]
where $M(0)=M$. If $\eps=1/4$ we will simply refer to
$\tmix(n)\equiv \tmix(n,1/4)$ as the mixing time. Using Remark \ref{rem:1} we immediately get
that 
\begin{equation}\label{lb1}
\tmix(n,\eps)\ge \max_{j=2,\dots, n}\tmix^{\rm East}(j,\eps),
\end{equation}
where $\tmix^{\rm East}(j,\eps)$ is the $\eps$-mixing time of the East
process at density $1/2$ on $[j]$. For the latter the following
sharp result was recently established by E. Lubetzky and the authors:
\begin{theorem}[\cite{GLM}]
\label{th:main2}
For any $p\in (0,1)$
there exists a positive constant $v=v(p)$ such that the East process on
$[n]$ at density $p$ exhibits cutoff at $v^{-1}n$ with an $O(\sqrt{n})$-window. More precisely,
for any fixed $0<\eps<1$ and large enough $n$,
\begin{align*}
 \tmix^{\rm East}(n,\eps)&= v^{-1} n + O\left(\Phi^{-1}(1-\eps)\, \sqrt{n}\right),
\end{align*}
where $\Phi$ is the c.d.f.\ of $\cN(0,1)$. 
\end{theorem}
By \eqref{lb1} it follows that $\tmix(n,\eps)\ge v_*^{-1} n +
O\left(\Phi^{-1}(1-\eps)\, \sqrt{n}\right)$ where $v_*\equiv v(1/2)$.  
A comparable \emph{upper bound}, a much harder task, was proved using
a very elegant argument by Y. Peres and A. Sly:
\begin{theorem}[\cite{peressly}]
\label{peressly} $\tmix(n) = \Theta(n)$.
\end{theorem}
Two natural questions arise at this point:
\begin{enumerate}
\item Is ${\limsup_n \tmix(n)/n= v_*^{-1}}$ ?  
\item Does $\tmix(n,\eps)$ exhibit the cutoff phenomenon similar to the
  single column process (\ie the East process) ?
\end{enumerate}
The main purpose of this article is to prove a positive answer to the above
questions for certain finite dimensional projections (e.g. for
the marginal process of the last $k$ columns, where $k$ is either
bounded in $n$ or it
grows very slowly). We also provide certain numerical findings which suggest that perhaps $\limsup_n \tmix(n)/n > v_*^{-1}$. 

Before stating our results more formally we recall that in order to
have the cutoff phenomenon  the \emph{relaxation
  time} $\trel(n)$ of the chain, defined as the inverse spectral gap of
the Markov generator, must satisfy $\trel(n)=o(\tmix(n))$ as $n\to
\infty$ (cf. \cite{LPW}). That MW satisfies this condition follows from work of Stong \cite{Stong} as mentioned previously. However we prove a stronger comparison result showing that in fact  $\trel(n)=\trel^{\rm East}(n-1)$,
where $\trel^{\rm East}(n-1)$ is the relaxation time of the East process
on $[n-1]$ at density $1/2$. The latter was shown to be bounded 
in $n$ (cf. \cite{AD02} or \cite{BFMRT}). 
\subsection{Main results}
\label{sec:main results}
For our arguments it will be convenient to generalize our initial
setting. For $1\le i_1< i_2 < \ldots < i_k\le n$ let
\[
\cM^{(n)}_{[i_1, \ldots, i_k]} =\{A\in \bbF_2^{n \times k}: A(i,j)=0, \text{ for } i>i_j,
A(i_j,i_j)=1\},
\]
so that $\cM_n= \cM^{(n)}_{[1,2,\ldots,n]}.$ As before we will simply
write $\cM_{[i_1, \ldots, i_k]}$ for $\cM^{(n)}_{[i_1, \ldots, i_k]}$ if
no confusion arises. Observe that for any $M\in \cM_n$ the sub-matrix  obtained from $M$ by retaining only
the columns $i_1,i_2,\dots, i_k$ belongs to $\cM_{[i_1, \ldots, i_k]}$.
 
Using the above observation, it is not hard to see that  the definition of
the random walk $M(t)$ on $\cM_n$  naturally extends to a continuous time ergodic
reversible Markov chain on $\cM_{[i_1, i_2,\ldots,
  i_k]}$, with reversible measure the uniform measure $\pi_{[i_1,i_2,\ldots,i_k]}$ on $\cM_{[i_1, i_2,\ldots, i_k]}$. We will denote by 
${\tmix}(n;\ [i_1,i_2,\ldots,i_k], \eps)$ and ${\trel}(n;\
[i_1,i_2,\ldots,i_k])$ its mixing time and relaxation time respectively. 
Our two main results read as follows.
\begin{maintheorem}
\label{thm:trel}
For any $1\le i_1< i_2 < \ldots < i_k\le n$ 
\[
{\trel}(n;\ [i_1,i_2,\ldots,i_k])=\trel^{\rm East}(i_k-1).
\]
In particular the relaxation time of the process $\{M(t)\}_{t\ge 0}$ on $\cM_n$
coincides with that of the East process on $[n]$ at density $1/2$.    
\end{maintheorem}
\begin{maintheorem}\label{finitecutoff}
There exists $C>0$ such that, given $k$ and $\eps\in (0,1)$, for all large enough $n$ and all $1\le i_1< i_2 <
\ldots < i_k\le n$ such that $i_k= n$,  
$${\tmix}(n,\eps;\ [i_1,i_2,\ldots,i_k])\le n/v_*+ Ck 2^k\sqrt{n}\log^{2}(n).$$
\end{maintheorem}
The $\eps$ dependence in the statement of the theorem is hidden in the choice of large enough $n$. A comparable lower bound $\tmix(n,\eps)\ge n/v_* +
O\left(\Phi^{-1}(1-\eps)\, \sqrt{n}\right)$ was stated right after Theorem \ref{th:main2}.

\begin{remark} Note that the dependence on $k$ is exponential and
  hence one can only hope to push this method of proof to prove sharp
  mixing,  up to  $k=\log (n)/2$ (ignoring smaller order terms)
  dimensional projections. The theorem holds also for $i_k=(1-o(1))n$
  but then the error term beyond $n/v_*$ will depend on the detailed
  form of the error term $o(1)$.
\end{remark}
\begin{remark}
Our method can be easily extended to cover the case when the entries
of the matrices belong to the field $\bbF_q$, with $q$ a prime, and
the matrix walk is as follows (cf.  \cite{peressly}): at rate one,
the $(i+1)^{th}$-row is multiplied by a uniformly chosen element of
$\bbF_q$ and added to the $i^{th}$-row where $i\in [1,2,\ldots,n-1]$. In this setting, for any
column, the projection obtained by denoting each
non zero entry by $1$, performs an East process with density $(q-1)/q$
\cite{peressly}*{Section 2.1}.  In this case we have the analog of the
above theorems with the relaxation time $\trel^{\rm East}(i_k-1)$ and
the velocity $v^*$ replaced by the corresponding quantities for the
new East process.  
\end{remark}
\subsection{Numerical Results:} \label{numerics}
In \cite{GLM} it was proved that the front (the rightmost one in the East process, (see \eqref{front}) for formal definition) behaves like a random walk with  velocity $v_*$ and has gaussian concentration. 
The figure above simulates the front process.
\begin{figure}[t]
\centering
\begin{tabular}{cc}
\includegraphics[width=.5\textwidth]{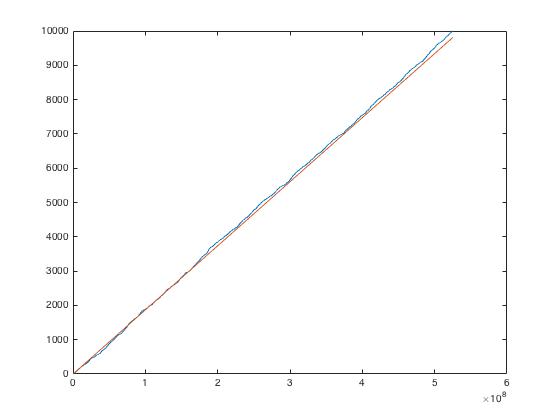} & 
\includegraphics[width=.5\textwidth]{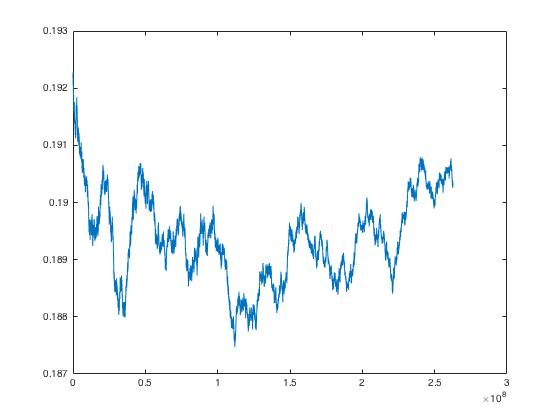}\\
(i) & (ii)
\end{tabular}
\caption{(i) The location of the front  of an East process (the rightmost one) for $p=\frac12$ en-route from $1$ to $10^4$ and the straight line approximating its trajectory under discrete time dynamics where at every discrete time the east process dynamics is performed at an uniformly chosen site in $[10^4]$. (ii) Average velocity of the front (location/time) along its journey (on a discrete time interval of order $10^8$: note that the discrete time dynamics on an interval on length $10^4$ is slower than the continuous time dynamics by a factor $10^4$). The above figures are scaled taking this into account.}
\label{fig:front}
\end{figure}

We ran an experiment on Matlab where the MW is run starting from the $1000 \times 1000$ identity matrix.  The statistic we keep track of is the rank of the rectangular block comprised of rows $[1,\ldots, 333]$ and columns  $[747, \ldots ,1000]$. It is well known that a rectangular block of size $m \times n$ with $m<n$ and independent ${\rm{Ber}}(1/2)$ entries has rank $m$ with probability approximately $1- 2^{n-m}.$ Thus our block only has has a chance of around $2^{-20}$ of having rank less than $333$ under equilibrium. The above figure shows the velocity  of the front to be in the window $[.185,.195]$ (which implies that the mixing time in discrete time steps  for the last column of the matrix is at most $(1/.18) 10^6$ taking into account fluctuations). We ran the MW for $(1/.15) 10^6$ discrete time steps and found the rank to be less than $200$ for five independent rounds.  The numbers are chosen to ensure that $(1/.15) 10^6$ is indeed larger than $(1/.18)10^6$ taking into account fluctuations. 
\section{Main ideas and organization of the article}
In this section we briefly describe the key ideas behind the proofs of the main results and how the article is organized. The proof of Theorem \ref{finitecutoff} is based on a recursive argument.  The case $k=1$ (the East process) is the content of \cite{GLM}. Assuming the theorem for $k$ we sketch how to prove it for $k+1.$ It suffices to consider the matrix $A=\cM_{[n-k, n-k+1, \ldots, n]},$ (see \ref{sec:main results}).
Note that the column evolutions are highly correlated and at a high level the approach is to find combinatorial situations which help us decouple the columns: for eg: if at any particular time $t$ when the $i^{th}$ row of $A$ is $[\underset{k \text{ times}}{\underbrace{0,\dots,0}},1],$ i.e.
$$\bigl({\textbf M}_{n-k}(t;i),\dots,{\textbf M}_{n}(t;i)\bigr)=(0,\ldots,0,1),$$ the $(i-1)^{th}$ row is updated by adding the $i^{th}$ row times a Bernoulli(1/2) variable $\xi_{i}(t)$, then forever in the future the random bit $\xi_{i}(t)$ only affects the last column and is independent of the first $k$ columns.  We then crucially use the underlying linearity in the dynamics (see discussions after Remark \ref{linear34}) to show that under the assumption that the following marginals on collection of $k$ columns, 
$$[\textbf {M}_{n-k},\dots,{\textbf M}_{n-1}], [\textbf {M}_{n-k},\dots,{\textbf M}_{n-2}{\textbf M}_{n}], [\textbf {M}_{n-k},\dots,{\textbf M}_{n-2},{\textbf M}_{n-1}+{\textbf M}_{n}]$$  where the addition is mod $2$, are well mixed, such patterns occur frequently at various rows which are not too far apart.  
Thus, at all the clock rings at such rows at such times, the last
column receives random bits which are independent from the first $k$
column evolutions (see discussion after \eqref{ringtimes1}). In order to use
these bits to prove that indeed the last column mixes independently of the
first $k$ columns, we rely on a local version of the argument appearing
in \cite{peressly}. The heart of this argument is based on the crucial
observation that a natural adjoint process to the upper triangular
matrix walk also has the law of an East process.  This fact, coupled with concentration results on the movement of the front (the rightmost one of the East process in infinite volume) obtained by refining results appearing in \cite{GLM}, is enough to conclude the proof. 

Theorem \ref{thm:trel} is proved by determining the asymptotic exponential decay rate of the total variation distance from equilibrium measure and relating it to the so called `persistence function' of the one dimensional East process. We then use a crucial input from \cite{chleboun}, which relates the persistence function to spectral gap. This is done in Section \ref{proof345}.

In Section \ref{prelim203} we define rigorously the underlying noise space, the primal and the adjoint processes, and collect various results which are crucially used, in particular results about combinatorial patterns helping mixing are discussed in Section \ref{pattern} and concentration results are discussed in Section \ref{sec:East}. The observation about the adjoint process relating it to the East process is discussed in Section \ref{adpr}. The local mixing results and the complete proof of Theorem  \ref{finitecutoff} appears in Section \ref{local245}.
\section{Preliminary results}\label{prelim203} 
In this section we collect several results that will be used for the
proof of the main results. 
\subsection{ Graphical construction, primal and adjoint maps}\label{formal12}
To
each $x\in \{1,\ldots, n-1\}$ we associate a rate one Poisson process and,
independently, a family of independent $\frac 12$-Bernoulli  random variables
$\{\xi_{x,k} : k \in \N\}$. The occurrences of the Poisson process
associated to $x$ will be denoted by $\{t_{x,k} : k \in \N\}$ and we
assume independence as $x$ varies. That fixes the
probability space $\O_n$. Notice that almost surely all
the occurrences $\{t_{x,k}\}$ as $x$ varies are different.

On $\O_n$ we construct a Markov chain with
state space $\cM_n$ according
to the following rules. Starting from $M\in \cM_n$ and given $\o\in \O_n$, at each time $t=t_{x,k},$ we update the state
of the chain $M(t^-)$ to
$M(t)=\left(\bbI+\xi_{x,k}E^{x,x+1}\right) M(t^-)$, i.e.  we add to
the $x^{th}$-row of $M(t^-)$ the $(x+1)^{th}$-row multiplied by the 
Bernoulli variable $\xi_{x,k}.$ It is easy to check that the above rule
defines a Markov chain $\{M(t;\o)\}_{t\ge 0}$ whose generator
coincides with that defined in \eqref{eq:gen}.

 Given $\omega\in \O_n$, a \emph{row} interval
 $I=[i,i+1,\dots,j]\subset [n]$ and a \emph{time} interval $\D=[s,t]$,
 $0<s<t,$ we denote by $\o_{I}(\D)$, the part of $\omega$ (\ie
 the clock rings and the Bernoulli variables) restricted to indices in
 the interval $I$ and clock rings in the time interval $\D$. More explicitly
\begin{equation}\label{variables}
\o_{I}(\D):=\left\{(t_{x,k}, \xi_{x,k}):\  x\in I, t_{x,k} \in \D \right\}.
\end{equation}
If $I=[n]$ we will simply write $\o(\D)$. We will denote by $\cF_{I}(\D)$ the $\s$-algebra generated by
the above variables and by $\tau_1 < \tau_2<\ldots \tau_m $ the
successive rings of $\o_{I}(\D)$. The position and the
Bernoulli variable corresponding to the $k^{th}$-ring $\t_k$ will be
denoted $x_k$ and $\xi_k$ respectively.  
Later on, we will need an extension of this space
where the finite set $[n]$ is replaced by the whole lattice $\bbZ$. In that case
the interval $I$  can be any finite interval in $\bbZ$.

%
\begin{definition}[Primal and adjoint map]
\label{prim-adj}
Given a row and time intervals $I,\D$ and $\o\in \O_n$,
we define the primal map
$\Phi_{I,\D}(\o,\cdot):\cX_n\mapsto \cX_n$ by    
\begin{equation}
\label{primal}
\Phi_{I,\D}(\omega, {\textbf Y})= (\bbI
+\xi_{m}E^{x_{m},x_{m}+1})\cdot \dots \cdot (\bbI
+\xi_{1}E^{x_{1},x_{1}+1}){\textbf Y}.
\end{equation}
The adjoint map $\Phi^*_{I,\D}(\o,\cdot): \cX^\top_n \to \cX^\top_n$ is
given instead by
\begin{equation}
  \label{dual}
\Phi^*_{I,\D}(\o,{\textbf Y}^\top)= {{\textbf Y}}^\top (\bbI
+\xi_{m}E^{x_{m},x_{m}+1})\cdot \dots \cdot (\bbI
+\xi_{1}E^{x_{1},x_{1}+1}).
\end{equation}
The maps $\Phi_{I,\D}(\o,\cdot),\Phi^*_{I,\D}(\omega,\cdot)$ can be naturally extended to act on
any matrix $M\in \cM_n$ by setting the $i^{th}$-column of
$\Phi_{I,\D}(\omega, M)$ equal to $\Phi_{I,\D}(\omega, {\textbf M}_i)$
and  the $i^{th}$-row of
$\Phi^*_{I,\D}(\omega, M)$ equal to $\Phi^*_{I,\D}(\omega, {\textbf
  R}^\top_i),$ where ${\textbf R}_1^\top,\dots,{\textbf R}^\top_n$ are the row vectors of $M$.
\end{definition}
\begin{remark} \label{linear34}
Note that the adjoint map is appropriately named since for any ${\textbf Z}^{\,\top},{\textbf Y}$,
\begin{equation}\label{eastdual}
{\textbf Z}^{\,\top}\phi_{I,\D}(\o,{\textbf Y}) =\phi^*_{I,\D}(\o,{\textbf Z}^{\,\top})\cdot {\textbf Y}.
\end{equation}
\end{remark}
If $I=[n]$ and $\D=[0,t]$ then, for any $M\in \cM_n\,$,
    $\Phi_{I,\D}(\o,M)=M(t;\o)$ where $M(t;\o)$ is the matrix
    determined by the graphical construction such that $M(0;\o)=M$. In
    particular the MW has the remarkable linearity feature,
    namely the evolution starting from $M+M'\in \cX_n$ is the sum (in $\cX_n$) of
    the evolutions starting from $M$ and $M'$ separately.
    
The matrix $\phi^*_{I,\D}(\o,M)$ coincides instead with the
matrix $\hat M(t;\o)$ constructed as follows. Let 
$m=m(\o,t)$ be the number of rings in $\o$ before $t$ and define the
adjoint randomness $\o^*=\o^*_t$ as the collection of times $t^*_{x,k}:=t-t_{x,m+1-k}$,
$k\le m$, and Bernoulli$(1/2)$ variables
$\xi^*_{x,k}:=\xi_{x,m+1-k}$. Then,
starting
from $\hat M(0;\o)= M$, at each time $t^*_{x,k}$ the matrix $\hat
M(t_{x,k}^*;\o^*)$ is updated by adding $\xi^*_{x,k}\ \times \
(x_{m+1-k})^{th}$-column to the $(x_{m+1-k}+1)^{th}$-column. 
For consistency we think that the ring at time $t^*_{x,k}$ is associated
to the the $(x_{m+1-k}+1)^{th}$-column. 
Clearly,
for any $i\in [n]$ and any $s\le t$, the $i^{th}$-row of $\hat M(s;\o^*)$ has the
first $i-1$ entries fixed equal to $0$, the $i^{th}$-one equal to $1$
and the other ones evolving as the East process on $[n-i-1]$ in the
time interval $[0,t]$. 

The graphical construction together with the linearity of the MW
allows the following representation of the marginal process on
e.g. the last column, a cornerstone for  \cite{peressly} and also for
us (cf. the proof of Proposition \ref{localmixing1} in Section
\ref{localmixingproofs}).

Fix $i\in [n-1]$ and a set of columns $\cC\subset [n]$. Without
loss of generality we assume that the last column of $\cC$ is the
$n^{th}$-one. Given a time interval $[t_1,t_2]$ let 
\begin{equation}\label{ringtimes1}
t_1< \tau_1 < \tau_2 < \ldots \tau_\nu< t_2 < \tau_{\nu+1}
\end{equation}
 be those rings $t_{i,k}$ in the graphical construction such
 that $t_{i,k}\in [t_1,t_2]$ and the $(i+1)^{th}$-entries of the
 columns in $\cC$ satisfy some a priori condition (e.g. they are  all equal to zero except the last one). Let $\xi_j\equiv \xi_{i,\tau_j}$ be the Bernoulli variable associated
to $\tau_j$ and let $\hat \cF_t$, $t\ge t_2$, be the $\s$-algebra
generated by all the variables generating $\cF_{[1,n]}([0,t])$ (see \eqref{variables}) except
the variables $\{\xi_j\}_{j=1}^\nu$. Let also $\a_j$ be the indicator
function that $M_{i+1,n}(\tau_j)=1$. Then
\begin{equation}
  \label{eq:linear}
  {\textbf M}_n(t;\o)= {\textbf A}_0(t;\o) +\sum_{j=1}^\nu \a_j(\o) \xi_j(\o) {\textbf A}_j(t,\o),
\end{equation}
where ${\textbf A}_0(t;\o)={\textbf M}_n(t;\hat \o)$ with $\hat\o$ obtained
from $\o$ by removing all the rings in \eqref{ringtimes1} at the $i^{th}$-row between time $t_1$ and
$t_2$ and, for any $j\in[\nu]$ such that $\a_j=1$,
${\textbf A}_j(t;\o)=\Phi_{I,\D_j}(\o,{\textbf e}_i)$ with $I=[1,i-1]$ and $\D_j=(\tau_j,t]$. The important feature of
$\{({\textbf A}_0,{\textbf A}_j,\a_j)\}_{j=1}^\nu$ is that they \emph{do not depend} on the Bernoulli variables $\{\xi_j\}_{j=1}^\nu$. 

The last, purely deterministic, result exploiting the interplay between the primal and
adjoint processes is as
follows (cf. the proof of Proposition 3.2 in \cite{peressly}). 
\subsubsection{The adjoint process argument}
\label{adpr} Given
$\o\in \O_n$ and $t>0$, choose an interval
$I=[a,b]\subset [n]$ together with a subset $t_1<t_2<\dots< t_k$ of the
Poisson times $\{\t_i\}_{i=1}^{m(t)}$ such that each $t_i$ corresponds
to an update of some row in $I$.
Define ${\textbf X}_j$ to be the restriction to $I$ of the vector
$\Phi_{I,[t_j,t]}(\o,{\textbf e}_b)$ (i.e. the column formed by the co-ordinates in $I$). The vectors
$\{{\textbf X}_j\}_{j=1}^k$ span the whole space $\cX_{b-a+1}$ iff for all
non-zero row vector ${\textbf Y}^\top\in \cX_{b-a+1}^\top$ there
exists $j\in [k]$ such that  ${\textbf Y}^\top \cdot {\textbf X}_j =1$. Using the adjoint
map we can rewrite ${\textbf Y}^\top \cdot {\textbf X}_j$ as follows. 

Let ${\textbf Z}^\top_x={\textbf Y}_x$
for $x\in I$ and ${\textbf Z}^\top_x=0$ otherwise. Then
\[
 {\textbf Y}^\top \cdot {\textbf X}_j ={\textbf Z}^\top \cdot \Phi_{I,[t_j,t]}(\o,{\textbf e}_b)=
 \Phi^*_{I,[t_j,t]} (\o,{\textbf Z}^\top)(b).
\]
Using the connection between $\Phi^*_{I,[t_j,t]} (\o,{\textbf Z}^\top) $ and the East process on the rows (cf. the discussion after \eqref{eastdual}), we conclude that
the vectors $\{{\textbf X}_j\}_{j=1}^k$ span the whole space iff for
all row vector ${\textbf Z}^\top$ which is identically equal to zero
outside $I$ there exists $j\in[k]$ such that the (graphical representation of) East process on $[n]$, evolving with the ``adjoint'' randomness $\o^*$ and initial condition $\textbf Z$, at time $t-t_j$ is equal to $1$ at the vertex $b$.


\subsection{Chernoff bounds for continuous time  Markov chains}
We give here a slight generalization of \cite{Lezaud}*{Theorem 3.4} proving Chernoff bounds for the
time average of continuous time Markov chains. 
\begin{lemma}
\label{LDPthm} Let $\{w(t)\}_{t\ge 0}$
  be a  continuous time finite Markov chain with reversible measure
  $\mu$. Assume that $\mu$ is positive and that the chain has a
  positive spectral gap $\g$. Let $A\subset [0,t]$ be of the form
$A=\cup_{i=1}^N [s_i,t_i]$, where $t_{i-1}\le s_i<t_i$, and set $|A|=\sum_i (t_i-s_i)$.
Then for any $\d\in (0,1)$ and any $f$ with $||f||_{\infty}\le 1,$ $\mu(f)=0, \mu(f^2)\le b^2,$
\begin{equation}
\label{LDP1}
\max_{w_0}\bbP_{w_0}\left[\bigl|\int_{0}^{t} f(w(s)){\textbf 1}(s\in
  A)ds\bigr|\ge \d |A|\right] \le \frac{2}{\mu_{\min}} \exp\left(-\gamma \delta^2|A|/(1+2b)^2 \right),
\end{equation}
where $\mu_{\min}= \min_{w\in \O}\mu(w)$ and $\bbP_{w_0}(\cdot)$ is the law of the chain starting from $w_0$. 
\end{lemma}
\begin{proof}For any event $E$ in path space $ \max_{w_0}\bbP_{w_0}(w(\cdot)\in E)\le \frac{1}{\mu_{\min}}
 \bbP_\mu(w(\cdot)\in E)$, so that
it suffices to prove that 
\begin{equation}
  \label{LDP2}
\bbP_\mu\left(\int_0^t ds\,f(w(s)){\textbf 1}(s\in
  A)  \ge \delta
  |A|\right)\le \exp\left(-\gamma \delta^2|A|/(1+2b)^2 \right). 
\end{equation}
The exponential Chebyshev inequality gives 
\[
 \bbP_\mu\left(\int_0^t ds\,  f(\h(s)){\textbf 1}(s\in
  A) \ge \delta
  |A|\right)\le e^{-\lambda \delta |A|}{\bbE}_\mu\left(e^{\lambda \int_0^t ds\, f(\h(s)){\textbf 1}(s\in
  A)}\right),\quad \l>0. 
\]
Let $H_\lambda$ be the self-adjoint operator on $\mathcal H=L^2(\mu)$
given by $H_\lambda= \mathcal L +\lambda f$, where $\cL$ is the
generator of the chain. If $\langle \cdot,\cdot\rangle$
denotes the scalar product in $\mathcal H$ then
\[
{\bbE}_\mu\left(e^{\lambda \int_0^t ds\, f(\h(s)){\textbf 1}(s\in
  A)}\right)=
\langle 1, e^{s_1\mathcal L }e^{(t_1-s_1) H_\lambda}e^{(s_2-t_1)\mathcal L}\dots e^{(t_N-s_N) H_\lambda}e^{(t-t_N)\mathcal L}1\rangle.
\]
Since $\|e^{s\mathcal L}\|=1$ we get that 
\[
{\bbE}_\mu\left(e^{\lambda \int_0^t ds\, f(\h(s)){\textbf 1}(s\in
  A)}\right)\le \exp\left(|A|\beta(\lambda)\right),
\]
where $\beta(\lambda)$ denotes the supremum of the spectrum of
$H_\lambda$.

We now bound from above $\beta(\lambda)$. For any function $h\in
\mathcal H$ with $\langle h,h\rangle=1$ we write $h=\alpha 1 +g$, so
that $\pi(g)=0$ and $\langle g,g\rangle=1-\a^2$. Then
\begin{gather*}
\langle h, H_\lambda h \rangle= \langle g,\mathcal L g\rangle +
\lambda\langle g,f g\rangle +2\alpha \lambda \langle 1,f g\rangle\\
\le -\gamma (1-\alpha^2)+ \l \sqrt{1-\a^2}+ 2b \l\a \sqrt{1-\a^2}\le \frac{\l^2(1+2b)^2}{4\g},
  \end{gather*}
which implies $\beta(\lambda)\le \frac{\l^2(1+2b)^2}{4\g}$. In conclusion 
\[
 \bbP_\mu\left(\int_0^t ds\,f(w(s)){\textbf 1}(s\in
  A)  \ge \delta
  |A|\right)\le e^{-\lambda \delta |A| +
  \frac{\l^2(1+2b)^2}{4\g}|A|}\le \exp\left(-\gamma \delta^2|A|/(1+2b)^2 \right),
\]
once we choose $\l= 2\d\g/(1+2b)^2$.
\end{proof}


\subsection{Concentration and mixing time for the East process}
\label{sec:East}
As we explained before the East process coincides with the marginal
process of each column ${\textbf M}^{(j)}(\cdot)$. It is therefore
important to develop a detailed analysis of its mixing time $\tmix^{\rm
  East}$. In this section we collect some refinements of the results
of \cite{GLM}*{Theorem 1}.  


We begin by considering the natural extension of the East process on
$\bbZ$ which is the unique Markov process on $\G=\{0,1\}^\bbZ$ with
(formal) generator\footnote{Notice that, for the reader's convenience,
  we adopt the convention that the $1$'s are the facilitating vertices,
  as it is the case for the upper triangular matrix walk, and not the
  $0$'s as it is
  customary in the East model literature.}  
\[
(\cL^{\rm East}f)(\eta)=\sum_{x\in
  \bbZ}\h_{x-1}\left[\,\bbE_x(f)(\eta)-f(\eta)\,\right],
\] 
where $\bbE_x(\cdot)$ denotes expectation over the variable
$\eta_x\in \{0,1\}$ with the fair Bernoulli distribution (see \cite{GLM}
for the formal definitions). 
In $\G$ we consider the set $\G^*$ of those  configurations such that the
variable 
\begin{equation}\label{front}
X(\eta):=
\sup\{x:\eta_x=1\}
\end{equation}
 is finite.
In the sequel, for any $\eta\in \G^*$  we will refer to $X(\eta)$ as
the \emph{front} of $\eta$. 

In order to understand the behavior of the process behind the front it
is convenient to move to an evolving reference frame in which the
front is always at the origin. More precisely let $\G_\Front$ denote
the set of configurations such that $X(\eta)=0$ and let
$\eta^\Front(\cdot)$ be the Markov process on $\G_\Front$
constructed as follows. The process behind the front evolves as the
usual East process. If the East dynamics tries to
move (by $\pm 1$) the front way from the origin then the whole configuration is shifted by $\mp 1$
in order to keep  to keep its front at the origin.
Blondel~\cite{Blondel} showed  that the law of $\eta^\Front(t)$
as
$t\to +\infty$ converges to an invariant measure $\nu$ on $(-\infty,-1]\cap \bbZ$ whose marginal on $(-\infty,-N]\cap \bbZ$ approaches
exponentially  fast (in $N$) the same marginal of the Bernoulli$(1/2)$ product measure. In the same limit she also
proved that $\frac1{t}
X(\eta(t)) $ converges in probability to $v_*$, where $v_*$ was defined right after Theorem \ref{th:main2}.
In \cite{GLM} Blondel's result was improved by establishing a
CLT for the variables $(X(\eta(t))-v_*t)/\sqrt{t})$ as $t\to +\infty$.

Using concentration properties for sums of  weakly
mixing sequence of random variables in Appendix \ref{concentration} we sharpen the CLT result of
\cite{GLM} as follows.
\begin{lemma} \label{concentration_front}There exists  universal constants $c,C$ such that for any $a> C\sqrt{t}\log^{3/2}(t),$ and $\eta(0) \in \G_{\Front},$
\begin{equation}\label{tailbound}
\bbP(|X(\eta(t))-v_*t|\ge a)\le e^{-c(\frac{a^2}{t})^{1/3}}.
\end{equation}
\end{lemma}
\begin{corollary}\label{onecutoff} For any $\g>0$, all $\eps \ge
  1/n^\g$ and $n$ large enough,
  \begin{equation}
    \label{eq:2}
\tmix^{\rm East}(n,\eps) \le n/v_*+ \sqrt{n}\log^2(n).
  \end{equation}
\end{corollary}

\begin{proof}[Proof of the Corollary] 
A natural coupling argument of \cite{GLM} (cf. the beginning of
Section 4.2.
there) proves that,  for any $\eta\in \G^F$ and any $\eps\in (0,1)$, 
$$
d_{TV}(\eta(t),\pi) \le \bbP_{\o}(X(\eta(t))<n).
$$
The bound \eqref{eq:2} is now an immediate consequence of \eqref{tailbound}.
\end{proof}
\subsection{Special row patterns and their large deviations}
\label{pattern}
For reason that will appear clearly in the proof of Theorem \ref{th:main2}, given a set of $k$ columns of the matrix $M(t)$ it will be important to have a good control on the statistics at time $t$ of the number of rows where the entries of the chosen columns exhibit certain special patterns (e.g. they are all equal to $0$ or only the last one is equal to $1$, etc). This is the main motivation for what follows. 

Consider a $k$ tuple of random vectors $({\textbf X}_1,\ldots,{\textbf X}_k)$ with joint distribution $\cP$ and let 
\begin{equation}
  \label{eq:N}
\cN=\sum_{x=1}^{n} \textbf{1}\left(({\textbf X}_1(x),\dots,{\textbf X}_k(x))=(0,\ldots, 0,1)\right). 
\end{equation}
Thus $\cN$ counts the number of coordinates where all but the last vector are zero.
\begin{lemma}Let $\pi_{j}$ be the uniform measure on $\cX_n^j$. Then for any positive integer $k,$
\label{lem:mixing} 
\begin{align*} 
\cP(|\cN-n/2^k|\ge n/2^{k+1})&< 6 e^{-n/(9\, 2^{2k-1})}\\&+
d_{TV}(({\textbf X}_{1},\dots,{\textbf X}_{k-1}),\pi_{k-1})\\
&+
 d_{TV}(({\textbf X}_{1},\dots,{\textbf X}_{k-2},{\textbf X}_{k}),\pi_{k-1})\\
&+ d_{TV}(({\textbf X}_{1},\dots,{\textbf X}_{k-2},{\textbf X}_{k-1}+{\textbf X}_{k}),\pi_{k-1})\\
&+d_{TV}(({\textbf X}_{1},\dots,{\textbf X}_{k-2}), \pi_{k-2}).
\end{align*}
\end{lemma}
\begin{proof}  
For the ease of reading, we first discuss the case $k=2$.
Let $\cN_1,\cN_2,\cN_3$ be defined as the variable $\cN$ in
\eqref{eq:N} but with other possible patterns
of $({\textbf X}_{1}(x),{\textbf X}_{2}(x) )$ as specified below:
\begin{align}
\label{fig1}
\cN_1\leftrightarrow (1,0),\quad \cN_2\leftrightarrow (0,0),\quad \cN_3\leftrightarrow (1,1).
\end{align}
Clearly $\cN_1+\cN_2+\cN_3+\cN=n$ so that $|\cN-n/4|\le \sum_{i=1}^3
|\cN+\cN_i-n/2|/2$. In particular 
\begin{equation}\label{azumabnd1}
\cP(|\cN-n/4|\ge n/8)\le \sum_{i=1}^3 \cP(|\cN+\cN_i-n/2|\ge n/12).
\end{equation}
The following is a
standard consequence of the Azuma- Hoeffding's inequality  and
the definition of total variation distance. For any random vector
${\textbf X}\in \cX_n$  distributed according to $\mathcal{P}$ and $a>0$
\begin{equation}
\label{azuma1} 
\mathcal{P}(|\sum_{x=1}^{n}{\textbf X}(x)-n/2|\ge a/2)\le
2e^{-\frac{a^2}{2n}}+d_{TV}({\textbf X},\pi).
\end{equation}
This implies: 
\begin{align*}
\cP(|\cN+\cN_2-n/2|\ge a/2)& \le 2e^{-\frac{a^2}{2n}}+d_{TV}({\textbf X}_{1},\pi) \\
\cP(|\cN+\cN_3-n/2|\ge a/2)& \le 2e^{-\frac{a^2}{2n}}+d_{TV}( {\textbf X}_{2},\pi) \\
\cP(|\cN+\cN_1-n/2|\ge a/2)& \le 2e^{-\frac{a^2}{2n}}+d_{TV} ({\textbf X}_{1}+{\textbf X}_{2},\pi).
\end{align*}
Choosing $a=n/6$ and summing the above three inequalities we finally get by \eqref{azumabnd1}
\begin{equation}
  \label{eq:aux}
\cP(|\cN-n/4|\ge n/8)\le 6e^{-n/72} +\sum_{i=1}^3 d_i,
\end{equation}
where $d_1,d_2,d_3$ are the total variation distance terms above.
 
For general $k,$ we define $\cN_1,\cN_2, \cN_3$ as before according to
the following patterns of the array $({\textbf X}_{1}(x),\dots,{\textbf X}_{k}(x) )$. The first $k-2$ elements are equal to zero while the
last two elements $({\textbf X}_{k-1)}(x),{\textbf X}_{k}(x))$ are as in \eqref{fig1}. Thus 
\[
\cN \leftrightarrow (\underbrace{0,\dots,0}_{k-2},0,1),\; \cN_1 \leftrightarrow (\underbrace{0,\dots,0}_{k-2},1,0),\; 
\cN_2 \leftrightarrow (\underbrace{0,\dots,0}_{k-2},0,0),\; 
\cN_3 \leftrightarrow (\underbrace{0,\dots,0}_{k-2},1,1).
\] 
In this case $\cN+\sum_{i=1}^3\cN_i$ is not deterministic; however by triangle inequality we still get  
\[
|\cN-n/2^{k}|\le \frac 12 \left(\sum_{i=1}^3|\cN+\cN_i-n/2^{k-1}| +|\cN+\sum_{i=1}^3 \cN_i -n/2^{k-2}|\right).
\] 
Using again \eqref{azuma1} we have: 
\begin{align*}
\cP(|\cN+\cN_2-n/2^{k-1}|\ge a/2)& \le 2e^{-\frac{a^2}{2n}}+d_{TV}(({\textbf X}_{1},\dots,{\textbf X}_{k-1}),\pi_{k-1})\\
\cP(|\cN+\cN_3-n/2^{k-1}|\ge a/2)& \le 2e^{-\frac{a^2}{2n}}+d_{TV}(({\textbf X}_{1},\dots,{\textbf X}_{k-2},{\textbf X}_{k}),\pi_{k-1})\\
\cP(|\cN+\cN_1-n/2^{k-1}|\ge a/2)& \le 2 e^{-\frac{a^2}{2n}}+ d_{TV}(({\textbf X}_{1},\dots,{\textbf X}_{k-2},{\textbf X}_{k-1}+{\textbf X}_{k}),\pi_{k-1})\\
\cP(|\cN+\sum^3_{i=1} \cN_i -n/2^{k-2}|\ge a/2)& \le 2 e^{-\frac{a^2}{2n}}+ d_{TV}(({\textbf X}_{1},\dots,{\textbf X}_{k-2}), \pi_{k-2}).
\end{align*}
We finally get the thesis as in the $k=2$ case by summing the above
inequalities computed for $a=\frac 13 \frac{n}{2^{k-1}}$. \end{proof}
\subsubsection{Application to the matrix walk}
\label{sec:appl MW}
Let $\mu$ be a probability measure on $\cM_n$ and consider the MW $M(t)$ with initial measure $\mu$ and law $\bbP_\mu(\cdot)$. Given a set of $k$ columns with indices $i_1< i_2 <\dots< i_k$  we denote by $\mu_1,\mu_2,\mu_3,\mu_4$ the marginals of $\mu$ on
\[
({\textbf M}_{i_1},\dots ,{\textbf M}_{i_k-1}),\  ({\textbf M}_{i_1},\dots ,{\textbf M}_{i_{k-2}},{\textbf M}_{i_k}),\  ({\textbf M}_{i_1},\dots, {\textbf M}_{i_{k-1}}+{\textbf M}_{i_k}),\ ({\textbf M}_{i_1},\dots, {\textbf M}_{i_{k-2}}),\] respectively. 

\begin{definition}\label{good}
Given a time interval $\D=[t_1,t_2]$ and $k, n \in \N$, we say that
the $x^{\rm th}$-row is
\emph{good for $\D$} if 
\[
\int_\D\textbf{1}\Bigl(\bigl(\textbf{M}_{i_1}(t;x),\dots,{\textbf M}_{i_k}(t;x)\bigr)=(0,\ldots,0,1)\Bigr)dt \ge \frac{t_2-t_1}{2^{k+1}},
\]
otherwise we term it \emph{bad for $\D$}.  We will sometimes suppress the reference
to $\D$ if clear from the context. 
\end{definition} 
Recall that $\pi_{j}$ is the uniform measure on $\cX_n^j.$
The result that will play a crucial role in getting a sharp bound on the mixing time of the marginal of the MW on $k$ columns is as follows.  
\begin{lemma}\label{goodindex}
There exists a universal constant $c>0$ such that for any time
interval $[t_1,t_2]$, any $k, n \in \N$ and any probability measure $\mu $ on $\cX_n^k$  
\begin{gather*}
\bbP_{\mu}(\text{all rows are bad for $[t_1,t_2]$}) \\
 \le \exp\left(-c(t_2-t_1) \right)+ d_{TV}(\mu_1,\pi_{k-1})+d_{TV}(\mu_2,\pi_{k-1})+d_{TV}(\mu_3,\pi_{k-1})+d_{TV}(\mu_4,\pi_{k-2}). 
\end{gather*}
\end{lemma}
\begin{proof} Fix $I=[t_1,t_2],k,n,\mu$ and recall the notation of Lemma \ref{lem:mixing} and of its
  proof. For $s\in I$ let $\cN(s),\ \cN_1(s),\ \cN_2(s),\ \cN_3(s)$ be the
  random variables defined after \eqref{eq:aux} computed for the
  vectors ${\textbf M}_{i_1}(s),\dots,{\textbf M}_{i_k}(s)$ and let for $a>0$ 
\begin{align*}
Y_1(s) &= \textbf{1}(|\cN(s)+\cN_1(s)-n/2^{k-1}|\ge a/2)\\
Y_2(s) &= \textbf{1}(|\cN(s)+\cN_2(s)-n/2^{k-1}|\ge a/2)\\
Y_3(s) &= \textbf{1}(|\cN(s)+\cN_3(s)-n/2^{k-1}|\ge a/2)\\
Y_4(s) &= \textbf{1}(|\cN_1(s)+\cN_2(s)+\cN_3(s)+\cN(s)-n/2^{k-2}|\ge a/2).
\end{align*}
When $\mu$ is the uniform measure $\pi$ we denote the expectation $
\bbE_\pi(Y_i(s))$ by $p_i$. Notice that $p_i$ does not depend on $s$ because of
reversibility w.r.t. $\pi$ and that $p_i \le
2 e^{-a^2/2n}$ because of \eqref{azuma1}. Let also 
\[
J_i=\int_I Y_{i}(s)ds, \quad i=1,\dots, 4,
\]
Using \eqref{LDP2} together with the fact that the spectral gap of the
chain $A(\cdot)$ is positive uniformly in $k,n$ (cf. Theorem \ref{thm:trel}
and \cites{Stong,peressly}), it follows that there exists $c>0$ such that, for $i=1,2,3$ and any $b_i\in (0,1)$,
$$\bbP_ \mu \left(|J_i-p_i t| \ge b_i t\right) \le \exp\left(-c{t
    b_i^2 }\right)+ d_{TV}(\mu_i,\pi_{k-1}),$$
where $t=t_2-t_1$.  
On the event that $J_{i} < (p_i+b_i)t$ for all $i=1,\ldots 4,$ the
fraction of times $s$ in $\D$ such that $Y_i(s)=0$ for all $i=1,\dots,4$  is at least $1-\sum_{i=1}^{4}(p_i+b_i).$
 For all such times  $s,$ a simple analysis shows that $|\cN(s)-
 \frac{n}{2^k}| \le a$. Therefore 
$$\int_\D \cN(s) \ge t(n/2^{k}-a)(1-\sum_{i=1}^{4}(p_i+b_i)).$$
 Hence by definition of $\cN$ (cf. \eqref{eq:N}), there exists $x\in
 [n]$ such that 
 \begin{gather*}
\int_\D\textbf{1}\Bigl(({\textbf M}_{i_1}(s;x),\dots,{\textbf M}_{i_k}(s;x))=(0,0,\ldots,0,1) \Bigr)ds \\
\ge
\frac{t}{n}(n/2^{k}-a)(1-\sum_{i=1}^{4}(p_i+b_i)).
 \end{gather*}
Choosing $a$ to be $n/(3.\, 2^{k})$
and  all the $b_i=\frac{1}{17}$, we see that
$\sum_{i=1}^{4}(p_i+b_i)\le 1/4$ and hence $x$ is good for $I$. 
\end{proof}

\section{Proof of Theorem \ref{thm:trel}}\label{proof345}
\begin{proof}For notational convenience and without loss of generality
  we consider only the full matrix walk (\ie
  $(i_1,\dots,i_k)=[n]$). Let us denote the spectral gaps of the MW
  $M(t)=({\textbf M}_1(t),\dots, \textbf{M}_n (t))$ and $\textbf{M}_n
  (t)$ (the last column which as mentioned before has the law of the
  East process with $p=1/2$) by $\lambda_1(n)$ and $\lambda_2(n)$
  respectively. Notice that $\l_2(n)$ is a non-increasing sequence
  (cf. \cite{CMRT}). That $\lambda_2(n)\ge \lambda_1(n)$ follows from the fact
  that $\textbf{M}_n (t)$ is a projection chain of $M(t).$ We now
  prove the other direction by combining the approach introduced in
  \cite{peressly} together with a crucial input from \cite{chleboun}.

Choose $i=n-1$ and $[t_1,t_2]=[0,t]$ in  \eqref{eq:linear} to write 
\[
{\textbf M}_n(t;\o)= {\textbf A}_0(t;\o) +\sum_{j=1}^\nu \a_j(\o) \xi_j(\o) {\textbf A}_j(t,\o)
 = {\textbf A}_0(t;\o) +\sum_{j=1}^\nu \xi_j(\o) {\textbf A}_j(t,\o),
\]
where for the last equality we use the fact that, as $M_{n,n}(t)=1$
deterministically, $\a_j=1\ \forall j\in [n]$.
Define the event $\mathcal C_n=\mathcal C_n(t)$ that the vectors
$\{{\textbf A}_j(t,\o)\}$ restricted to the first $n-1$ rows span the
vector space $\cX_{n-1}.$ Using Lemma \ref{span}, on the event
$\cC_n,$ we can couple $\textbf{M}_n (t)$ to a vector
distributed according the equilibrium measure independent of
$[\textbf{M}_1 (t),\textbf{M}_2 (t),\ldots, \textbf{M}_{n-1} (t)].$
Thus 
$$d_{TV}(M(t),\pi) \le d_{TV}(({\textbf M}_1(t),\dots,
\textbf{M}_{n-1} (t)),\pi)+\mathbb{P}(\cC_n^c),
$$
\ie 
$$d_{TV}(M^{(n)}(t),\pi)\le \sum_{i=1}^n\bbP(\cC^c_i).$$
The key new input is now the following result.
\begin{lemma}
\label{persi}
For each $n\ge 2$, $\lim_{t\to \infty}\frac 1t \log(\bbP(\cC_n^c))\ge -\l_2(n)$.  
\end{lemma}
Assuming the lemma and recalling that $\l_2(i)\ge \l_2(n)$ for $i\le
n$, we get immediately that $\lim_{t\to \infty}\frac 1t
\log(d_{TV}(M(t),\pi))\le -\l_2(n)$. A standard functional analysis
argument now implies $\l_1(n)\ge \l_2(n)$. 
\end{proof}
\begin{proof}[Proof of Lemma \ref{persi}]
For the  stationary East process
$\sigma(t)=[\sigma_1(t),\sigma_2(t),\ldots \sigma_n(t)]$ on $[n],$ let
$\tau$ be the first time there is a legal ring at $n$, i.e. the clock
rings at $n$ and $\sigma_{n-1}(\tau)=1.$ 
In \cite{chleboun}*{Theorem A.1} it was proved that the
exponential decay rate of $\bbP_\pi(\t\ge t)$ is governed by the spectral gap
$\lambda_2$ :
\[
\lim_{t\to \infty} \frac{\log \bbP_\pi(\t\ge t)}{t}=-\lambda_2.
\]
Recall the notation from \eqref{eq:linear} that $0< \tau_1<\tau_2<\ldots<\tau_{\nu} <t$ are the ring times for the $n-1^{th}$ row. Thus in the ``adjoint'' randomness $\o^*,$ $$0< t-\tau_\nu<t-\tau_{\nu-1}<\ldots<t-\tau_{1} <t$$ are the ring times for the $n^{th}$ bit.
Now by the discussion of Section \ref{adpr},  choosing $I=[1,n-1]$ it
follows that the event $\cC_n$  occurs iff for any  row vector
${\textbf Z}^\top$ there exists $j\in[\nu]$ such that the (graphical
representation of) East process on $[2,\ldots, n]$, evolving with the adjoint
randomness $\o^*$ and initial condition $\textbf Z^\top$ (note that the first coordinate is frozen), at time
$t-\tau_j$ is equal to $1$ at the vertex $n-1$ for some $1\le j\le
\nu$, and hence  in the terminology of the East process, the ring at $t-\tau_j$ is legal for the $n^{th}$ bit. 
Thus putting the above together along with a shift of coordinate to transfer $[2,n]$ to $[n-1]$ we get 
\[
\bbP(\cC^c(n))\le \sum_{{\textbf Z}^\top\in \cX_{n}}\bbP_{{\textbf
  Z}^\top}(\tau >t)\le 2^{2n}\bbP_\pi(\t>t) 
\]
and the lemma follows.
\end{proof}

\section{Local mixing and the proof of Theorem \ref{finitecutoff}.}\label{local245}
In this section we first prove that certain patterns of bits as in Section
\ref{pattern} facilitate some form of local mixing of the
matrix walk (cf. Proposition  \ref{localmixing1} below). We then prove that
these patters are likely to occur (cf. Lemma \ref{good10}) and, as a consequence, we are able to establish a precise
recursion relation between the mixing times of the marginal of the MW on $k$
and $k-1$ columns (cf. Lemma \ref{mixtheorem}). Putting all together
we will finally deduce Theorem \ref{finitecutoff}.
\subsection{Local mixing}
\label{locmix}We first need some initial setup. Recall the
definition  given in Section \ref{sec:main results} of the space
$\cM_{[i_1,\dots,i_k]}$ as the set of the columns $1\le i_1<i_2<\ldots
<i_k\le n$ and let, for any $g: \cM_{[i_1,\ldots,i_k]} \mapsto [-1,1]$ and any subset
$\cI$ of
entries of the chosen columns,  $g_{\cI}$ be
the $\pi$-average of $g$ over the entries in $\cI$. 

Fix a burn-in time $t_1>0$ together with a positive time lag $\D=o(n)$
and let $t_2=t_1+\kappa\D$, where $\kappa$ is a large constant
(possibly depending on $k\,$\footnote{The constant $\kappa$  will be taken to be a large universal constant times
  $2^k$ for large $k$.}) to be fixed later on. Recall Definition
\ref{good} and let, for any $a\le n-2\D$, $\cB_a\equiv \cB_a(t_1,\D)$
be the event that there is a row $x\in [a+\D,a+2\D]$ which is good for
the time interval $[t_1,t_2]$. Finally let $\cI\equiv \cI_{i_k,a,\D}=\{i_k\}\times [1,a+\D]$.

In the above setting the local mixing result reads as follows:
\begin{proposition}\label{localmixing1} There exists a universal constant $c>0$ and a choice of $\kappa=\kappa(k)>0$ such that the following holds. For any $g: \cM_{[i_1,\ldots,i_k]} \to [-1,1]$ which does not depend on the entries $\{(i_k,j): 1\le j< a  \}$, any $t>t_2$  and any initial condition $M$ of the matrix walk:
 $$\left|\E \bigl[g(M(t))-g_{\cI}(M(t))\bigr]\right| \le 2[e^{-c\D}+\bbP(\cB_a^c)].$$
\end{proposition}
The error term $\bbP(\cB_a^c)$ can be bounded from above using Lemma \ref{goodindex} provided that the burn-in time $t_1$ is large enough. Recall the definition
of $\tmix(n;
  [i_1,\dots,i_k],\eps)$ as the $\eps$-mixing time of the marginal matrix walk on the columns
  $i_1,i_2,\dots,i_k$ and let 
\[
\tmix(n,k,\eps)= \max_{i_1,\dots,i_k}\tmix(n;
  [i_1,\dots,i_k],\eps). 
\]
\begin{lemma}\label{good10}  There exists a universal constant $c$ such that, for any $\eps\in (0,1)$ and any $a\le n-2\D$, the following holds. If $t_1\ge\tmix(n,k-1,\eps)$ then
$$
\bbP(\cB_a^c)\le 4\e+e^{-c\D}.
$$
\end{lemma}
Before proving the lemma and the proposition we
prove Theorem \ref{finitecutoff}. 

\subsection{Proof of Theorem \ref{finitecutoff}} 
At the basis of the argument there is the following recursive relation between the mixing times of $k$ and $k-1$ columns.
\begin{lemma}\label{mixtheorem} There exists $c>0$ such that the
  following holds. Fix $\eps\in
  (0,1)$ and let $\eps'=\eps/(9\sqrt{n})$. Then, for any
  given $k\in \N$ and $n$ large enough, 
$$
\tmix(n,k,\eps) \leq \tmix(n,k-1,\eps')
+ c.2^k\, \sqrt{n}.
$$
\end{lemma}
Assuming the lemma we proceed as follows. Fix $k,n$ together with
$\eps\in (0,1)$ and let $\eps'= 9^{-k}n^{-k/2}\eps$. 
Using Corollary \ref{onecutoff} we get
$$\tmix(n,1,\eps') \le n/v_*+ \sqrt{n}\log^2(n)$$
for all $n$ large enough depending on $\eps,k$.
Now applying Lemma \ref{mixtheorem} $k-1$ times starting from $k=1$ gives 
$$\tmix(n,k,\eps) \le n/v_* + \sqrt{n}\log^2(n) +ck2^{k} \sqrt{n}.$$ 
\qed

\begin{proof}[Proof of Lemma \ref{mixtheorem}]
Given $g:\cM_{[i_1,\dots,i_k]} \mapsto \bbR$ such that 
  $\|g\|_\infty\le 1$, let $g^{(j)}$ be defined
  recursively as follows: 
\begin{align*}
 g^{(0)}(\cdot)=g(\cdot),\qquad
  g^{(j)}(\cdot)=g_{\cI_{k,j}}^{(j-1)}(\cdot),\quad j\le N,
\end{align*}
where $N=n/\D$ and $\cI_{k,j}=\{i_k\}\times R_j$ with $\D=\lfloor
\sqrt n\rfloor$ and $R_j=((j-1)\D,j\D]$.
Since $g^{(j)}$  is the $\pi$-average of $g^{(j-1)}$
w.r.t. to the entries of the $i_k^{th}$-column in the interval $R_j$
we get in
particular that $g^{(j)}(\cdot)=g_{\{i_k\}\times [1,j\D]}(\cdot)$ and
$g^{(N)}$ is the $\pi$-average of $g$ over the entries on the $i_k^{th}$-column.

A simple telescopic sum together with the triangle inequality gives  
\begin{gather}\label{telescoping1}
|\bbE_{A}\left[g(M(t))-\pi(g)\right]|\nonumber \\
 \le  |\bbE_{M}\left[g^{N}(M(t))-\pi(g)\right]|+ \sum_{j=0}^{N-1} |\bbE_{M}\left[g^{(j)}(M(t))-g^{(j+1)}(M(t))\right]|.
\end{gather}
Notice that $g^{N}(\cdot)$ depends only on the entries
in the $i^{th}_1,\dots,i^{th}_{k-1}$-columns so that  
 \[
|\bbE_{M}\left[g^{N}(M(t))-\pi(g)\right]| \le
2d_{TV}((\textbf{M}_{i_1}(t),\dots,{\textbf{M}}_{i_{k-1}}(t)),\pi),
\]
where, with an abuse of notation,
$\pi$ denotes the uniform measure on the chosen $k$ columns of the initial matrix $M$.

Next we observe that for any $j \le N-1$  the function $g^{(j)}$
satisfies the hypothesis of Proposition \ref{localmixing1} if we
choose $a=j\D$.  
Thus, if $t=\tmix(n,\eps',k-1)+\kappa \D$ with $\kappa$ the
constant appearing in Proposition \ref{localmixing1}, we can appeal to
Proposition \ref{localmixing1} and Lemma \ref{good10} to get that
every term in the sum is bounded from above by  $8\e'+ 4 e^{-c\D}$ for some universal constant $c>0$.

In conclusion 
\begin{align*}
2d_{TV}((\textbf{M}_{i_1}(t),\dots,{\textbf{M}}_{i_{k}}(t)),\pi)&=\max_{\substack{g:\cM_{i_1,\dots,i_k} \mapsto \bbR \\
  \|g\|_\infty\le 1}}|\bbE_{M}\left[g(M(t))-\pi(g)\right]| \\& \le
 2d_{TV}((\textbf{M}_{i_1}(t),\dots,{\textbf{M}}_{i_{k-1}}(t)),\pi)+ N(8\e' +4e^{-c\D})\\
& \le 2\eps' + 8\sqrt{n}(\eps' +e^{-c\sqrt{n}})\le \eps,
\end{align*}
where the last two inequalities follow by the assumption that
$t>\tmix(n,\eps',k-1)$ and our choice of $\D$.
\end{proof}

\subsection{Local mixing: proofs}
\label{localmixingproofs}
Recall the notation defined at the beginning of the section.
\begin{proof}[Proof of Proposition \ref{localmixing1}] 
Fix $g: \cM_{[i_1,\ldots,i_k]} \to [-1,1]$ which does not depend on the entries $\{(i_k,j): 1\le j< a  \}$ and, for $i\in [a+\D,a+2\D]$, let $\cG_{i}$ be the event that $i$ is the
largest good index in $[a+\D,a+2\D]$ w.r.t. the time interval
$[t_1,t_2]$.

Observe that $\cG_i$ is measurable with respect to the $\s$-algebra $\cF_{[i,n]}([0,t_2])$ and hence
also w.r.t. $\cF_{[i,n]}([0,t])$ if $t\ge t_2$. 
Let 
\begin{equation}\label{ringtimes}
t_1< \tau_1 < \tau_2 < \ldots \tau_\nu< t_2 < \tau_{\nu+1}
\end{equation}
 be the set of  random times $t$ such that the clock of the
 $(i-1)^{th}$-row rings and the $i^{th}$-entries of the chosen columns are
 all zero except the last one (i.e. they form the pattern $[0,0,\dots,0,1]$). 

Let $\xi_j\equiv \xi_{i-1,\tau_j}$ be the Bernoulli variable associated
to $\tau_j$ and let $\hat \cF_t$ be the $\s$-algebra
generated by all the variables generating $\cF_{[2,n]}([0,t])$ except
the variables $\{\xi_j\}_{j=1}^\nu$. Recalling \eqref{eq:linear},
we have 
\begin{equation*}
  {\textbf M}_{i_k}(t;\o)= {\textbf A}_{0}(t;\o) +\sum_{j=1}^\nu \xi_j(\o) {\textbf A}_j(t,\o),
\end{equation*}
where ${\textbf A}_{j}(t;\o)=\Phi_{I,\D_j}(\o,{\textbf e}_i)$ with
$I=[1,i-1]$ and $\D_j=(\tau_j,t]$. Notice that the
 variables $\a_j$ appearing in \eqref{eq:linear} are all equal to one
 because of the definition of the times $\t_1,\dots,\t_\nu$. The
 fundamental property of this decomposition is that $\{({\textbf  A}_{0},{\textbf A}_{j})\}_{j=1}^\nu$ as well as all the other columns
 $({\textbf M}_{i_1},\dots, {\textbf M}_{i_{k-1}})$  
 are all  measurable with respect to $\hat \cF$. 

Let now $\cP$ be the subspace of $\cX_{i-a}$ spanned by the
restriction of the vectors $\{{\textbf A}_{j}(t;\o)\}_{j=1}^\nu$ to the
row interval $[a,a+1,\dots,i-1]$ and let $\cC_i=\cG_i\cap\{\o:\ \cP=\cX_{i-a}\}$. Clearly the events $\cC_i$ are disjoint. Let also $\cI_0=\{i_k\}\times [1,a+\D]$.
\begin{claim}
\begin{equation}\label{couple}
\left|\E\left[g(M(t))-g_{\cI_0}(M(t))\right]\right| \le 2\bigl[1-\sum_{i=a+\D}^{a+2\D}\bbP(\cC_i)\bigr].
\end{equation}
\end{claim}
\begin{proof}[Proof of the Claim]
Observe that, conditionally on $\hat \cF_t$, on the event $\cC_i$ all the entries ${\textbf M}_{i_k}(t;x), \ x \in [a,i-1]\}$, are i.i.d Bernoulli$(1/2)$ (cf. Lemma \ref{span}).  Thus, using the fact that $g$ does not depend on the entries ${\textbf M}_{i_k}(x),\ x<a,$ and letting $\cI_i=\{i_k\}\times [1,i-1]$, we get 
\begin{align}\label{conditioning10}
\E[g(M(t))\textbf{1}(\cC_i)|\hat \cF_t]& =g_{\cI_i}(M(t))\textbf{1}(\cC_i),\\
\nonumber
\E[g_{\cI_0}(M(t))\textbf{1}(\cC_i)|\hat \cF_t]& =g_{\cI_i}(M(t))\textbf{1}(\cC_i).
\end{align} 
In conclusion
\begin{gather}
|\E[g(M(t))-g_{\cI_0}(M(t))]|\nonumber \\
 =|\E\Bigl[\bigl[g(M(t))-g_{\cI_0}(M(t))\bigr]\, \bigl[\sum_{i=a+\D}^{a+2\D}\textbf{1}(\cC_i) +(1-\sum_{i=a+\D}^{a+2\D}\textbf{1}(\cC_i))\bigr]\Bigr]|\nonumber \\
=|\,\E[g(M(t))-g_{\cI_0}(M(t))][1-\sum_{i=a+\D}^{a+2\D}\textbf{1}(\cC_i)]\,|\nonumber \\
 \le 2(1-\sum_{i=a+\D}^{a+2\D}\bbP(\cC_i)).
\label{uffa}
\end{gather}
\end{proof}
The last step in the proof of Proposition \ref{localmixing1} is an upper bound on the r.h.s. of \eqref{uffa}. 
\begin{lemma}\label{conditional} There exists $n_0=n_0(k)$ such that, for all $n\ge n_o$ and all $i \in [a+\D,a+2\D],$ 
\[
\bbP(\cC_i| \cG_i)\ge 1- e^{-c\D},
\]
where $c>0$ is a universal constant. In particular
\[
1-\sum_{i=a+\D}^{a+2\D}\bbP(\cC_i)\le 1-(1-e^{-c\D})(\sum_{i=a+\D}^{a+2\D}\bbP(\cG_i))\le e^{-c \D} +\bbP(\cB_a^c).
\]
\end{lemma}
Clearly the lemma finishes the proof of Proposition \ref{localmixing1}.
\end{proof}
\begin{proof}[Proof of Lemma \ref{conditional}]
Fix $i \in [a+\D,a+2\D]$ together with $\o \in \cG_i$ and recall \eqref{ringtimes}. Using the \emph{adjoint process argument} described in Section
\ref{adpr}, $\o\in \cC_i$ if for any row vector $\mathbf
Z^\top\in \cX^\top_n$ such that $\mathbf Z^\top(x)=0$ for all $x\notin
[a+\D,i-1]$ there exist $j\in [\nu]$ such that the East process on
$[n]$, evolving with the
adjoint randomness $\o^*=\o^*_t$ (cf. the discussion after Remark
\ref{eastdual}) and starting from $\mathbf Z^\top$, at time $t-\t_j$ is
equal to one at the vertex $i-1$. It is
crucial for what follows that, for any given $s\in [t_1,t_2]$, the
event $\cE_{i-1}(s;\mathbf Z^\top)$ that the East
process using $\o^*$ and starting from $\mathbf Z^\top$, at time $t-s$ is
equal to one at the vertex $i-1$  involves only the
randomness $\xi_{x,k},t_{x,k}$ with $x\in [1,i-2]$ and $k$ such that
$t_{x,k}\in [s,t]$.
    
Let also $T_i(\o) \subset
[t_1,t_2]$ be the measurable subset of all times $s\in [t_1,t_2]$ such
that the restriction of the $i^{th}$-row of $M(s;\o)$ to the columns
$i_1,\dots,i_k$ form the pattern $\bigl(0,0,\ldots,0,1\bigr)$. Notice that $T_i(\cdot)$ is measurable
w.r.t. $\cF_{[i,n]}([0,t_2])$ and that $\t_j\in T_i\ \forall j\in
[\nu]$. Moreover $T_i$ is the union of intervals and, for $\o\in
\cG_i$,  its Lebesgue measure $|T_i(\o)|$ satisfies (cf. Definition \ref{good})
\[
|T_i(\o)|\ge (t_2-t_1)/2^{k+1}\ge \kappa \D/2^{k+1}.
\] 
Finally let $R_{i-1}(\o)=\{t_{i-1,k}\}_{k=1}^\infty$ be the rings of the Poisson clock of the $(i-1)^{th}$-row. 
A simple union bound over the at most $2^\D$ possible choices of the row vector
${\mathbf Z}^\top$ proves that the lemma follows if we can show that
\[
\max_{{\mathbf Z}^\top\text{ as above }}\bbP(\cE_{i-1}(s;{\mathbf Z}^\top) \text{ fails }\forall s\in T_i\cap R_{i-1} 
\tc
\cG_i)\le e^{-c'\D},
\]
where the constant $c'$ can be made as large as we want (in particular
larger than $\log 2$) by choosing the
constant $\kappa$ in $t_2=t_1+\kappa \D$ large enough.

Using the independence between the event $\cE_{i-1}(s;{\mathbf Z}^\top)$ and the
$\s$-algebra $\cF_{[i,n]}([t_1,t_2])$, the fact that $T_i$
is measurable w.r.t. $\cF_{[i,n]}([t_1,t_2])$ and \eqref{LDP1} of Theorem \ref{LDPthm}
applied to the East process, we obtain
\begin{gather*}
{\textbf 1}(\cG_i)\bbP\Bigl(|\{s\in [t_1,t_2]:\ \cE_{i-1}(s;{\mathbf Z}^\top) \text{ holds }\}\cap T_i|
\le |T_i|/4\tc
\cF_{[i,n]}([t_1,t_2])\Bigr)\\
\le {\textbf 1}(\cG_i)\bbP\Bigl(\Big|\int_{t_1}^{t_2}\ ds \bigl({\textbf
  1}(\cE_{i-1}(s;{\mathbf Z}^\top))-\frac 12\bigr) {\textbf 1}(s\in
T_i)\ \Big|\ge
|T_i|/2\ \big|\  \cF_{[i,n]}([t_1,t_2])\Bigr)
\\
\le 2^{\D}e^{-c (t_2-t_1)/2^{k+1}}=2^{\D}e^{-c \kappa\D/2^{k+1}},
 \end{gather*}
for some numerical constant $c>0$ related to the spectral gap of the
East process.

Finally, using again the independence of $R_i$ from $\cF_{[1,i-2]}([0,t])$ and $\cF_{[i,n]}([0,t])$ we get that 
\[
\bbP(\{s:\
E_{i-1}(s;{\mathbf Z}^\top)\text{ holds}\} \cap T_i\cap
R_i|=\emptyset\tc \cF_{[1,i-2]\cup[i,n]}([0,t]))\le 
e^{-|\{s:\
E_{i-1}(s;{\mathbf Z}^\top)\text{ holds}\} \cap T_i|}.
\]
In conclusion, for any ${\mathbf Z}^\top\text{ as above }$, 
\begin{gather*}
  \bbP(\cE_{i-1}(s;{\mathbf Z}^\top) \text{ fails }\forall s\in T_i\cap R_{i-1} 
\tc
\cG_i)\le 2^{\D}e^{-c \kappa\D/2^{k+1}}+ e^{- \kappa\D/2^{k+3}}
\end{gather*}
and the proof is finished by taking $\kappa$ to be a large enough constant times $2^k$.
\end{proof}

\begin{proof}[Proof of Lemma \ref{good10}]   
Using the definition of $t_1$ we have 
\begin{align*}
  d_{TV}\left(\left({\textbf M}_{i_1}(t_1),\dots, {\textbf
      M}_{i_{k-1}}(t_1)\right),\pi\right)&\le \e\\
d_{TV}\left(\left({\textbf M}_{i_1}(t_1),\dots, {\textbf
      M}_{i_{k-2}}(t_1)\right),\pi\right)&\le \e\\
d_{TV}\left(\left({\textbf M}_{i_1}(t_1),\dots, {\textbf
      M}_{i_{k-2}}(t_1),{\textbf M}_{i_k}(t_1)\right),\pi\right)&\le
                                                                  \e\\
d_{TV}\left(\left({\textbf M}_{i_1}(t_1),\dots, {\textbf
      M}_{i_{k-1}}(t_1)+{\textbf M}_{i_k}(t_1)\right),\pi\right)&\le \e.
\end{align*}
Lemma \ref{goodindex} and the definition of $t_2$ now imply that 
\[
\bbP(\cB^c) \le e^{-c(t_2-t_1)}+4\e\le e^{-c\D}+4\e.
\]
\end{proof}

\appendix
\section{}
We collect here some  result and proofs that were omitted from the main text.  We begin with a simple lemma on the distribution of random linear combinations in $\cX_n$
that we state here without proof.
\begin{lemma}\label{span} For any $n,k \in \N,$ given a set of  vectors ${\textbf V}^{(1)},\ldots,{\textbf V}^{(k)} \in \cX_n,$ and $k$ independent Bernoulli$(1/2)$ variables $\{\xi_1,\dots,\xi_k\}$, the random vector $\sum_{i=1}^{k}a_iv_i$ is distributed uniformly in $\cX_n$ iff ${\textbf V}^{(1)},\ldots,{\textbf V}^{(k)}$ span $\cX_n$. 
\end{lemma} 

\subsection{Proof of Lemma \ref{concentration_front}}\label{concentration}
We first recall one of the main technical results from \cite{GLM}
saying that, for
any $\h\in \G_\Front$, the increments in the position of the front (\eqref{front} )
(the variables $X$ below) behave asymptotically as a
stationary sequence of weakly dependent random variables with
exponential moments. In the sequel $\nu$ denotes the invariant measure of
the process as seen from the front (cf. Section \ref{sec:East}).

Define $\xi_n:= X(\h(n))-X(\h(n-1)),\ n\ge 1,$
so that
\begin{equation}
  \label{eq:19}
X(\h(t))=\sum_{n=1}^{N_t} \xi_n + \left[X(\h(t))-X(\h(N_t))\right], \quad N_t=
\lfloor t\rfloor.
\end{equation}


\begin{lemma}[\cite{GLM}*{Corollary 3.2}] 
\label{cor:wf} There exists $\alpha \in (0,1)$ and $\g>0$ such that
the following holds. Let $f:\bbR\mapsto [0,\infty)$ be such that $e^{-|x|}f^2(x)\in L^1(\bbR)$. Then
  \begin{equation}
    \label{eq:20}
C_f\equiv \sup_{\h\in
  \G_\Front}\bbE_\h\left[f(\xi_1)^2\right]<\infty,
  \end{equation}
and 
\begin{align}
\label{eq:20tris}
\sup_{\h\in \G_\Front}|\bbE_\h\left[f(\xi_n)\right]- \bbE_\nu\left[f(\xi_1)\right]|=
O(e^{-\g n^\a}) \quad \forall n\ge 1,
 \end{align}
where the constant in the r.h.s.\ of \eqref{eq:20tris} depend on $f$ only through the constant $C_f$.
%
\end{lemma}
The proof of Lemma \ref{concentration_front} will follow by constructing the appropriate Doob's martingale and using martingale concentration by bounding the martingale difference. 
For any probably space $(\O,\cF,\bbP)$ let $\{X_1,X_2,\ldots,  X_n\},$ be random variables taking values in a set $\cS.$ Consider any function $\phi: \cS^n\to \bbR$ and a filtration $$\{\O,\emptyset\}=\cF_0 \subset \cF_1 \subset \ldots \cF_k=\cF$$ 
The Doob's martingale is given by 
$M_{i}=\E(\phi(X)|\cF_i),$ where $X=(X_1,\ldots, X_n).$
We define  the martingale difference $$V_{i}(\phi)=M_i-M_{i-1},$$ and let  $D_i(\phi)=||V_{i}(\phi)||_\infty.$
Then by the  classic Azuma-Hoeffding martingale concentration 
$$\bbP(|\phi(X)-\E(\phi(X)|>r) \le 2e^{-\frac{r^2}{2\sum_{i=1}^k D^2_i}}.$$
To use the above, in our setting  let $k=n,\ \phi(X)=\sum_{i=1}^n X_i$ and let $\cF_{i}=\cF_{\Z}([0, i])$ (cf. \eqref{variables}) and $k=n$.
Ideally we would want to take $X_i=\xi_i.$ However these variables are
not quite bounded but, using \eqref{eq:20}, they have exponential tails. Thus we choose $$X_i = \min (\max(\xi_i, K),-K),$$ (the value of $K$ would be specified later). 
We now bound the martingale differences.
Note that by \eqref{eq:20tris} for any $j\ge i+1$ choosing $f(x)=\min (\max(x, K),-K),$
\begin{align*}
\E(X_j|\cF_{i})&=O(K\,e^{-(j-i)^{\alpha}})+\E_{\nu}(f(\xi_1)),\\
\E(X_j|\cF_{i-1})&=O(K\,e^{-(j-i+1)^{\alpha}}) +\E_{\nu}(f(\xi_1)).
\end{align*}
Thus 
\begin{align*}
|M_{i}-M_{i-1}| &\le |X_{i}|+C\,K\sum_{i=1}^{n}e^{-\g i^{\alpha}}\\
&\le CK.
\end{align*}
Thus putting everything together we get, 
$$\bbP(|\phi(X)-\E(\phi(X))|>r) \le 2e^{-\frac{r^2}{(CK)^2}}.$$
All that is left is to take care of the truncation. 
Let $\xi=(\xi_1,\ldots, \xi_n).$
Using \eqref{eq:20} 
$$|\E(\phi(X))-\E(\phi(\xi))|=O(ne^{-cK}).$$
Thus we get 
$$\bbP(|\phi(\xi)-\E(\phi(\xi))|>r) \le ne^{-K}+2e^{-\frac{r'^2}{2K^2n}}$$
where $r'=r-O(nKe^{-cK}).$ When $r>C\sqrt{n}\log^{3/2}(n)$ for some large constant $C$ we choose $K^{3}=r^2/n$ and see that this choice yields for some constant $c>0,$
$$\bbP(|\phi(\xi)-\E(\phi(\xi))|>r) \le e^{-c(\frac{r^2}{n})^{1/3}}.$$
Thus we see for any $n,$
$$\bbP(|X(w(t_n))-vt_n|>r) \le e^{-c(\frac{r^2}{n})^{1/3}}.$$
For a general and large enough, $t_n\le t<t_{n+1}$ and $r>C\sqrt{n}\log^{3/2}(n),$
\begin{align*}
\bbP(|X(w(t))-vt|>r) & \le \bbP(|X(w(t_n))-vt_n|>r/2) +  \bbP(|X(w(t))-X(w(t_n))|\ge r/4), \\ 
& \le  e^{-c(\frac{r^2}{n})^{1/3}} +e^{-cr},\\
& \le  e^{-c(\frac{r^2}{n})^{1/3}}. 
\end{align*}
\qed

\section*{Acknowledgements}
A part of this work was completed when both the authors were visiting the Simons Institute for the Theory of Computing during the semester program on Counting Complexity and Phase Transitions.
We thank the Institute for its wonderful research environment. SG's research is supported by a Miller Research Fellowship.


\bibliography{Upper_triangular_cutoff}
\bibliographystyle{plain} 

%
%
%
%
%
\end{document}